\documentclass[12pt]{amsart}%rticle}

\usepackage{latexsym}
\usepackage{amssymb}
\usepackage{stmaryrd}
\usepackage{enumerate}
\usepackage{graphicx}
\usepackage{psfrag}
\usepackage{xcolor}
\definecolor{green4}{HTML}{009000}
\definecolor{magenta4}{HTML}{900090}

\newcommand{\excise}[1]{}%{$\star$\textsc{#1}$\star$}

\newtheorem{thm}{Theorem}
\newtheorem{lemma}[thm]{Lemma}

\newtheorem{cor}[thm]{Corollary}
\newtheorem{prop}[thm]{Proposition}

\theoremstyle{definition}
\newtheorem{example}[thm]{Example}

\newtheorem{defn}[thm]{Definition}

        {\begin{list}
                {\noindent\makebox[0mm][r]{\arabic{enumi}.}}
                {\leftmargin=5.5ex \usecounter{enumi}}
        }
        {\end{list}}

        {\begin{list}
                {\noindent\makebox[0mm][r]{(\roman{enumi})}}
                {\leftmargin=5.5ex \usecounter{enumi}}
        }
        {\end{list}}

\def\bem#1{\textbf{#1}}

\def\AA{{\mathcal A}}
\def\CC{{\mathcal C}}

\def\aa{A}

\DeclareMathOperator{\NW}{\mathit{NW}}
\DeclareMathOperator{\SE}{\mathit{SE}}
\DeclareMathOperator{\up}{\mathrm{up}}
\DeclareMathOperator{\dn}{\mathrm{dn}}

\def\rp{\mathcal{RP}}

\def\implies{\Rightarrow}

\renewcommand\iff{\Leftrightarrow}

\def\sub#1#2{_{#1 \times #2}}

\font\co=lcircle10

\def\jr{\smash{\raise2pt\hbox{\co \rlap{\rlap{\char'005} \char'007}}
               \raise6pt\hbox{\rlap{\vrule height5pt}}
               \raise2pt\hbox{\rlap{\hskip4pt \vrule height0.4pt depth0pt
                width5.7pt}}
               \raise2pt\hbox{\rlap{\hskip-9.5pt \vrule height.4pt depth0pt
                width6.2pt}}
               \lower6pt\hbox{\rlap{\vrule height4.5pt}}}}
\def\tj{\smash{\raise2pt\hbox{\co \rlap{\rlap{\char'005} \char'007}}
               \raise6pt\hbox{\rlap{\vrule height3pt}}
               \raise2pt\hbox{\rlap{\hskip4pt \vrule height0.4pt depth0pt
                width2.7pt}}
               \raise2pt\hbox{\rlap{\hskip-6.5pt \vrule height.4pt depth0pt
                width3.2pt}}
               \lower5pt\hbox{\rlap{\vrule height3.5pt}}}}
\def\je{\smash{\raise2pt\hbox{\co \rlap{\rlap{\char'005}
                \phantom{\char'007}}}\raise6pt\hbox{\rlap{\vrule height5pt}}
               \raise2pt\hbox{\rlap{\hskip-9.5pt \vrule height.4pt depth0pt
                width6.2pt}}}}
\def\er{\smash{\raise2pt\hbox{\co \rlap{\rlap{\phantom{\char'005}} \char'007}}
               \raise2pt\hbox{\rlap{\hskip4pt \vrule height0.4pt depth0pt
                width5.7pt}}
               \lower6pt\hbox{\rlap{\vrule height4.5pt}}}}
\def\+{\smash{\lower6pt\hbox{\rlap{\vrule height17pt}}
                \raise2pt\hbox{\rlap{\hskip-9pt \vrule height.4pt depth0pt
                width18.7pt}}}}
\def\pl{\smash{\lower5pt\hbox{\rlap{\vrule height14pt}}
                 \raise2pt\hbox{\rlap{\hskip-6.5pt \vrule height.4pt depth0pt
                 width13.2pt}}}}
\def\hor{\smash{\raise2pt\hbox{\rlap{\hskip-9.5pt \vrule height.4pt depth0pt
                width19.2pt}}}}
\def\ver{\smash{\lower6pt\hbox{\rlap{\vrule height17pt}}}}

\def\textcross{\ \smash{\lower4pt\hbox{\rlap{\hskip4.15pt\vrule height14pt}}
                \raise2.8pt\hbox{\rlap{\hskip-3pt \vrule height.4pt depth0pt
                width14.7pt}}}\hskip12.7pt}
\def\textelbow{\ \hskip.1pt\smash{\raise2.8pt%
                \hbox{\co \hskip 4.15pt\rlap{\rlap{\char'005} \char'007}
                \lower6.8pt\rlap{\vrule height3.5pt}
                \raise3.6pt\rlap{\vrule height3.5pt}}
                \raise2.8pt\hbox{%
                  \rlap{\hskip-7.15pt \vrule height.4pt depth0pt width3.5pt}%
                  \rlap{\hskip4.05pt \vrule height.4pt depth0pt width3.5pt}}}
                \hskip8.7pt}

\def\qr{\smash{\lower6pt\hbox{\rlap{\vrule height8pt}}
                \raise2pt\hbox{\rlap{\hskip0pt \vrule height.4pt depth0pt
                width9.7pt}}}}

\begin{document}%%%%%%%%%%%%%%%%%%%%%%%%%%%%%%%%%%%%%%%%%%%%%%%%%%%%%%%%
%%%%%%%%%%%%%%%%%%%%%%%%%%%%%%%%%%%%%%%%%%%%%%%%%%%%%%%%%%%%%%%%%%%%%%%%

\mbox{}\vspace{2.9ex}
\title{Duality of antidiagonals and pipe dreams}

\author{Ning Jia}
\address{School of Mathematics, University of Minnesota, MN}
\email{njia@math.umn.edu}

\author{Ezra Miller}
\address{School of Mathematics, University of Minnesota, MN}
\thanks{EM was partially supported by NSF grants DMS-0304789 and DMS-0449102}
\email{ezra@math.umn.edu}

%\subjclass[2000]{Primary: ; Secondary:}

\date{}

%begin{abstract}
%end{abstract}

\maketitle

%First page headline in AmS-LaTeX for S\'eminaire Lotharingien de Combinatoire
%--first part
\thispagestyle{myheadings}
\font\rms=cmr8
\font\its=cmti8
\font\bfs=cmbx8
\markright{\its S\'eminaire Lotharingien de
Combinatoire \bfs 58 \rms (2008), Article~B58e\hfill}
\def\thepage{}
\markboth{\SMALL NING JIA AND EZRA MILLER}{\its S\'eminaire Lotharingien de
Combinatoire \bfs 58 \rms (2008), Article~B58e\hfill}

\noindent

The cohomology ring $H^*(Fl_n)$ of the manifold of complete flags in a
complex vector space~$\mathbb{C}^n$ has a basis consisting of the
Schubert classes~$[X_w]$, the cohomology classes of the Schubert
varieties~$X_w$ indexed by permutations $w \in S_n$.  The ring
$H^*(Fl_n)$ is naturally a quotient of a polynomial ring in $n$
variables; nonetheless, there are natural $n$-variate polynomials, the
Schubert polynomials, representing the Schubert
classes~\cite{LSpolySchub}.  The most widely used formulas
\cite{BJS,FSnilCoxeter} for the Schubert polynomial~$\mathfrak{S}_w$
are stated in terms of combinatorial objects called \emph{reduced pipe
dreams}, which can be thought of as subsets of an $n \times n$ grid
associated to~$w$.

Reduced pipe dreams are special cases of curve diagrams invented by
Fomin and Kirillov \cite{FK96}.  They were developed in a
combinatorial setting by Bergeron and Billey~\cite{BB}, who called
them \emph{rc-graphs}, and ascribed geometric origins in
\cite{koganThesis,grobGeom}.  One of the main results in the latter is
that the set~$\rp_w$ of reduced pipe dreams is in a precise sense dual
to a family~$\AA_w$ of simpler subsets of the $n \times n$ grid called
\emph{antidiagonals} (antichains in the product of two size~$n$
chains): every antidiagonal in~$\AA_w$ shares at least one element
with every reduced pipe dream, and each antidiagonal and reduced pipe
dream is minimal with this property \mbox{\cite[Theorem~B]{grobGeom}}.
The antidiagonals were identified there with the generators of a
monomial ideal whose zero set corresponds to a certain flat
degeneration of the Schubert variety~$X_w$.  Geometrically, the
duality meant that the components in the special fiber
% of the flat family
are in bijection with the reduced pipe dreams in~$\rp_w$, which yield
directly the monomial terms in~$\mathfrak{S}_w$.  It was pointed out
in \cite[Remark~1.5.5]{grobGeom} that the proof of this duality was
roundabout, relying on the recursive characterization of $\rp_w$ by
``chute'' and ``ladder'' moves \cite{BB}, along with intricate
algebraic structures on the corresponding monomial ideals; our purpose
here is to give a direct \mbox{combinatorial explanation}.

% \emph{antidiagonals} $\AA_w$ and \emph{reduced pipe dreams} (or
% \emph{rc-graphs}) $\rp_w$, both viewed as subsets of the $n\times n$
% grid.

% While $\AA_w$ has a simple, straightforward definition, $\rp_w$ is
% the subject of more combinatorial interests.

Fix a permutation $w \in S_n$, and identify it with its
\bem{permutation matrix}, which has an entry~$1$ in row~$i$ and
column~$j$ whenever $w(i) = j$, and zeros elsewhere.  We write $w\sub
pq$ for the upper left $p \times q$ rectangular submatrix of~$w$ and
\[
  r_{pq} = r_{pq}(w) = \#\{(i,j) \leq (p,q) \mid w(i)=j\}
\]
for the rank of the matrix~$w\sub pq$.  Let
\[
  l(w) = \#\big\{(i,j) \mid w(i) > j \text{ and } w^{-1}(j) > i\big\}
       = \#\big\{i < i'\mid w(i) > w(i')\big\}
\]
be the number of inversions of~$w$, which is called the \bem{length}
of~$w$.

\begin{defn}
A $k \times \ell$ \bem{pipe dream} is a tiling of the $k \times \ell$
rectangle by \bem{crosses}~$\textcross$ and \bem{elbows}~$\textelbow$.
A pipe dream is \bem{reduced} if each pair of pipes crosses at most
once.  
\end{defn}

For examples as well as further background and references, see
\cite[Chapter~16]{cca}.  Pipe dreams should be interpreted as ``wiring
diagrams'' consisting of pipes entering from the west and south edges
of a rectangle and exiting though the north and east edges, with the
tiles $\textcross$ and $\textelbow$ indicating intersections and bends
of the pipes.

The set~$\rp_w$ of reduced pipe dreams for a permutation~$w$ consists
of those $n \times n$ pipe dreams with $l(w)$ crosses such that the
pipes entering row $i$ from the west exit from column~$w(i)$.  In such
a pipe dream $D$, all of the tiles below the main
southwest-to-northeast (anti)diagonal are necessarily elbow tiles.  We
identify $D$ with its set of crossing tiles, so that $D \subseteq [n]
\times [n]$ is a subset of the $n \times n$ grid.

% \[
% \begin{tinyrc}{
%   \begin{array}{@{}|c|c|c|c|@{}}
% 	  \hline  \pl & \tj & \pl & \tj
% 	\\\hline  \tj & \tj & \tj & \tj
% 	\\\hline  \tj & \tj & \tj & \tj
% 	\\\hline  \tj & \tj & \tj & \tj
% 	\\\hline
%   \end{array}
% }\end{tinyrc}
% \]

\begin{defn}
An \bem{antidiagonal} is a subset $\aa \subseteq [n] \times [n]$ such
that no element is (weakly) southeast of another: $(i,j) \in \aa$ and
$(i,j) \leq (p,q)$ $\implies$ $(p,q) \notin \aa$.  Let $\AA_w$ be the
set of minimal elements (under inclusion) in the union over all $1
\leq p,q \leq n$ of the set of antidiagonals in~$[p] \times [q]$ of
size $1+r_{pq}(w)$.
\end{defn}

For example, when $w = 2143 \in S_4$,
%$w(1) = 2$, $w(2) = 1$, $w(3) = 4$, and $w(4) = 3$,
\begin{align*}
\AA_{2143} &=
  \Big\{\big\{(1,1)\big\},\big\{(1,3),(2,2),(3,1)\big\}\Big\}
\\
\text{and}\quad
\rp_{2143} &=
   \Big\{\big\{(1,1), (1,3) \big\},\big\{(1,1),(2,2)\big\}, \big\{(1,1),
   (3,1)\big\}\Big\}.
\end{align*}
As another example, when $w = 1432 \in S_4$, 
\[
\AA_{1432} = \Big\{\!
  \big\{(1,2),(2,1)\big\},
  \big\{(1,2),(3,1)\big\},
  \big\{(1,3),(2,1)\big\},
  \big\{(1,3),(2,2)\big\},
  \big\{(2,2),(3,1)\big\}
  \!\Big\}
\]
and
\begin{align*}
\rp_{1432} = \Big\{&
  \big\{(1,2), (1,3), (2,2)\big\},
  \big\{(1,2), (2,1), (3,1)\big\},
\\&
  \big\{(2,1), (2,2), (3,1)\big\},
  \big\{(1,2), (2,1), (2,2)\big\}
  \Big\}.
\end{align*}

Given any collection~$\CC$ of subsets of $[n] \times [n]$, a
\bem{transversal} to~$\CC$ is a subset of $[n] \times [n]$ that meets
every element of~$\CC$ at least once.  The \bem{transversal dual}
of~$\CC$ is the set~$\CC^\vee$ of all minimal transversals to~$\CC$.
(Our definition of transversal differs from that in matroid theory,
where a transversal meets every subset only once.  Here, our
transversals do not give rise to matroids: the transversal duals need
not have equal cardinality, so they cannot be the bases of a matroid.)
When no element of~$\CC$ contains another, it is elementary that
taking the transversal dual of~$\CC^\vee$ yields~$\CC$.

Our goal is a direct proof of the following, which is part of
\cite[Theorem~B]{grobGeom}; see also \cite[Chapter~16]{cca} for an
exposition, where it is isolated as Theorem~16.18.

%First page headline in AmS-LaTeX for S\'eminaire Lotharingien de Combinatoire
%--restoring the headers and pagenumbering
\pagenumbering{arabic}
\addtocounter{page}{1}
\markboth{\SMALL NING JIA AND EZRA MILLER}{\SMALL DUALITY OF ANTIDIAGONALS AND
PIPE DREAMS}

\begin{thm}\label{theorem}
For any permutation~$w$, the transversal dual of the set $\rp_w$ of
reduced pipe dreams for~$w$ is the set $\AA_w$ of antidiagonals
for~$w$; equivalently, $\rp_w = \AA^\vee_w$.
\end{thm} 
 
In other words, every antidiagonal shares at least one element with
every reduced pipe dream, and it is minimal with this property. 

\begin{proof}
We will show two facts.
\begin{enumerate}[\quad{Claim~}1.]
\item\label{claim1}
$D \in \rp_w \implies D \supseteq E$ for some $E \in \AA^\vee_w$.

\item\label{claim2}
$E \in \AA^\vee_w \implies E \in \rp_v$ for some permutation $v \geq
w$ in Bruhat order.
\end{enumerate}
Assuming these, the result is proved as follows.  First we show that
$\AA^\vee_w \subseteq \rp_w$.  To this end, suppose $E \in
\AA^\vee_w$.  Then $E \in \rp_v$ for some $v \geq w$ by Claim~2, so $E
\supseteq D$ for some $D \in \rp_w$ by elementary properties of Bruhat
order (use \cite[Lemma~16.36]{cca}, for example: reduced pipe dreams
for~$v$ are certain reduced words for~$v$, and each of these contains
a reduced subword for~$w$).  Claim~1 implies that $D \supseteq E'$ for
some $E' \in \AA^\vee_w$.  We get $E = E'$ by minimality of~$E$, so $E
= D$ and $v = w$.

To show that $\rp_w \subseteq \AA^\vee_w$, assume that $D \in \rp_w$.
Claim~1 implies that $D \supseteq E$ for some $E \in \AA^\vee_w$.  But
$E \in \rp_w$ by the previous paragraph, so $D = E$ because all
reduced pipe dreams for~$w$ have the same number of crossing tiles.

The remainder of this paper proves Claims~1 and~2.
\end{proof}

The key to proving Claims~1 and~2 is the combinatorial geometry of
pipe dreams.  For this purpose, we identify $[n] \times [n]$ with an
$n \times n$ square tiled by closed unit subsquares, called
\bem{boxes}.  This allows us to view pipes, crossing tiles, elbow
tiles, and pieces of these as curves in the plane.  We shall
additionally need the following.

\begin{defn}
A \bem{northeast grid path}\/ is a connected arc whose intersection
with each box is one of its four edges or else the rising diagonal
$\diagup$ of the box.
\end{defn}

\begin{example}\label{ex:adiagNE}
Fix an antidiagonal $A$ in the $k \times \ell$ rectangle $[k] \times
[\ell]$.  There exists a northeast grid path~$G$, starting at the
southwest corner of $[k] \times [\ell]$ and ending at the northeast
corner, whose sole $\diagup$ diagonals pass through the boxes in~$A$.
There might be more than one; a typical path~$G$ with $k = 7$, $\ell =
15$, and $|A| = 3$ looks as follows:
\[
\psfrag{to}{$\!\rightsquigarrow$}
\includegraphics{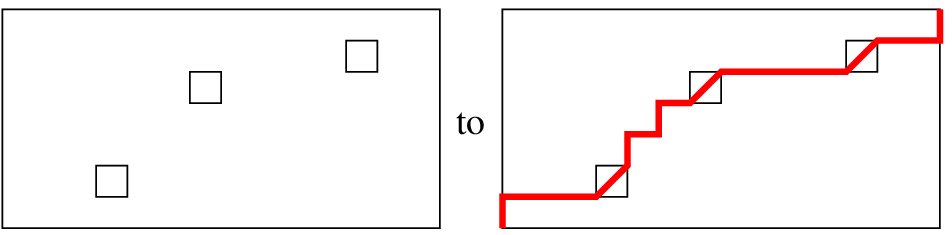}
\vspace{.2ex}
\]
\end{example}

\begin{example}\label{ex:pipeNE}
Let $P$ be a pipe in a pipe dream, or a connected part of a pipe.
Define $\up(P)$ to be the northeast grid path consisting of the north
edge of each box traversed horizontally by~$P$, the west edge of each
box traversed vertically by~$P$, and the rising diagonal in each box
through which $P$ enters from the south and exits to the east.
Dually, define $\dn(P)$ to consist of the south edge of each box
traversed horizontally by~$P$, the east edge of each box traversed
vertically by~$P$, and the rising diagonal in each box through which
$P$ enters from the west and exits to the north.
\begin{center}
%psfrag{up(P)}{\textcolor{magenta4}{$P$} with \textcolor{green4}{$\up(P)$}}
\includegraphics{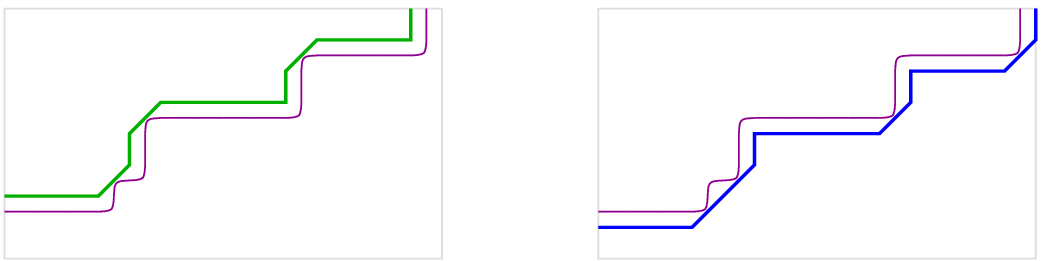}

\hfill \textcolor{magenta4}{$P$} with \textcolor{green4}{$\up(P)$}
\hfill \textcolor{magenta4}{$P$} with \textcolor{blue}{$\dn(P)$}
\hfill\mbox{}
\end{center}
\end{example}

Whenever a northeast grid path is viewed as superimposed on a pipe
dream, we always assume (either by construction or by fiat) that no
pipe crosses it vertically through a diagonal $\diagup$ segment.  This
is especially important in the next two lemmas.

The arguments toward Claims~1 and~2 are based on two elementary
principles for a region~$R$ bounded by northeast grid paths.  Such a
region has a lower (``southeast'') border $\SE = \SE(R)$ and an upper
(``northwest'') border $\NW = \NW(R)$.

\begin{lemma}[Incompressible flow]\label{l:incompFlow}
Fix a pipe dream.  If $k$ pipes enter $R$ vertically through $\SE$ and
none cross $\SE$ again, then $\NW$ has at least $k$ horizontal
segments.
\end{lemma}
\begin{proof}
Every pipe crossing $\SE$ vertically exits $R$ vertically
through~$\NW$.%
\end{proof}

Thus the ``flow'' consisting of the pipes entering from the south is
``incompressible''.
% The flow consists of the pipes entering from the south;
% incompressibility forces the vertical flow to exit out of at least as
% many horizontal segments as it required to enter.

\begin{lemma}[Wave propagation]\label{l:waveProp}
If none of the pipes entering $R$ vertically through $\SE$ cross $\SE$
again, then $\#\{\diagup$ segments in SE$\} \geq \#\{\diagup$ segments
in NW$\}$.
\end{lemma}
\begin{proof}
The sum of the numbers of horizontal and diagonal segments on $\NW$
equals the corresponding sum for $\SE$ since these arcs enclose a
region.  Now use Lemma~\ref{l:incompFlow}.
\end{proof}

The ``waves'' here are formed by the northwest halves of elbow tiles,
each viewed as being above a corresponding rising $\diagup$ diagonal;
see also the proof of Lemma~\ref{waterfront}.  In the proof of
Proposition~\ref{reduced}, the ``flipped'' version is applied: if none
of the pipes entering the region~$R$ vertically (downward) through
$\NW$ cross $\NW$ again, then \mbox{$\#\{\diagup$ segments in NW$\}
\geq \#\{\diagup$ segments in SE$\}$}.

\begin{prop}\label{p:incompFlow}
If $D \hspace{-1.1pt}\in\hspace{-1.1pt} \rp_w\hspace{-.2pt}$ has no
$\textcross$ on an antidiagonal $A \subseteq [p]
\hspace{-1pt}\times\hspace{-1pt} [q]$ \mbox{then $|A| \leq r_{pq}$}.
\end{prop}
\begin{proof}
The $q$ pipes in $D$ that exit to the north from columns $1,\ldots,q$
are of two types: $r_{pq}$ of them enter $[p] \times [q]$ horizontally
into rows $1,\ldots,p$, and the other $q - r_{pq}$ of them enter into
$[p] \times [q]$ vertically from the south.  Now simply apply the
principle of incompressible flow to the region bounded by a northeast
grid path as in Example~\ref{ex:adiagNE} and the path consisting of
the south and east edges of $[p] \times [q]$.
\end{proof}

\begin{cor}\label{Dcontainsf}
Every pipe dream $D \in \rp_w$ is transversal to~$\AA_w$, so Claim~1
holds.
\end{cor}
\begin{proof}
If an antidiagonal $A \subseteq [p] \times [q]$ lies in~$\AA_w$, then
by definition $A$ has size at least $1+r_{pq}(w)$.  Now use
Proposition~\ref{p:incompFlow}.
\end{proof}

\begin{lemma}\label{waterfront}
If $D\in \rp_v$ for some permutation~$v$, then for every $p,q \in
\{1,\ldots,n\}$, there is an antidiagonal of size $r_{pq}(v)$ in $[p]
\times [q]$ on which $D$ has only elbows.
\end{lemma}
\begin{proof}
Let $I_{pq}$ be the set of all $r_{pq}$ of the pipes in~$D$ that enter
weakly above row $p$ and exit weakly to the left of column~$q$.  For
each $k \leq q$, let $b_k$ be the southernmost box (if it exists) in
column~$k$ that intersects any $P \in I_{pq}$; otherwise, let $b_k$ be
the northernmost box in column $k$.  Of the $q$ pipes exiting to the
north from columns $1, \dots, q$, precisely $q-r_{pq}$ of them cross
some $b_k$ vertically from the south.  The remaining $r_{pq}$ of the
boxes $b_k$ must be elbow tiles, and these form the desired
antidiagonal.
\end{proof}
\begin{center}
\psfrag{ p,q}{\footnotesize $p,q$}
\includegraphics{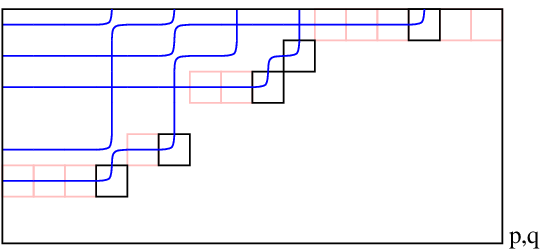}

%\textcolor{blue}
{The pipes in~$I_{pq}$ and the boxes $b_1,\ldots,b_q$ in the proof of
Lemma~\ref{waterfront}}
\end{center}

\begin{prop}\label{reduced}
Every transversal $E \in \AA^\vee_w$, thought of as a pipe dream, is
reduced.
\end{prop}
\begin{center}
\psfrag{1}{\footnotesize $\scriptstyle 1$}
\psfrag{2}{\footnotesize $\scriptstyle 2$}
\psfrag{P}{\footnotesize $P$}
\psfrag{b}{\footnotesize $b$}
% \psfrag{G}{\footnotesize $G$}
\psfrag{Q}{\footnotesize $Q$}
\includegraphics{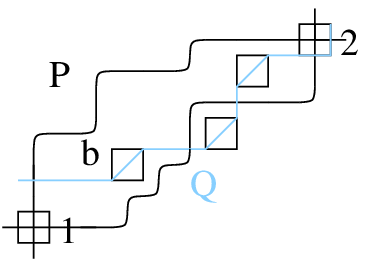}

{Illustration of the proof of Proposition~\ref{reduced}}
\end{center}

\begin{proof}
Fix a (not necessarily minimal) transversal $E$ of~$\AA_w$ containing
two pipes $P$ and $Q$ that cross twice, say at $\boxplus_1$
and~$\boxplus_2$, with $\boxplus_2$ northeast of $\boxplus_1$. Assume
that the pipes $P$ and $Q$ as well as the crosses $\boxplus_1$ and
$\boxplus_2$ are chosen so that the taxicab distance (i.e., the sum of
the numbers of rows and columns) between them is minimal. Then one of
the pipes, say $P$, is northwest of the other on the boundary of this
area. The minimality condition implies that no pipe in~$E$ crosses $P$
or~$Q$ twice, so the principle of wave propagation holds for any
region~$R$ such that $\SE(R)$ is part of~$\up(Q)$, and the flipped
version holds if $\NW(R)$ is part of~$\dn(P)$.

Our goal is to show that if $\boxplus_2$ is replaced by an elbow tile
in~$E$, then $E$ will still have a crossing tile on every antidiagonal
$\aa \in \AA_w$, whence the transversal $E$ is not minimal.  The
method: for any $\aa \in \AA_w$ containing~$\boxplus_2$, we produce a
new antidiagonal $\aa' \in \AA_w$ such that $\boxplus_2 \notin \aa'$,
and furthermore every box in $\aa'$ is either an elbow tile in~$E$ or
a crossing tile of~$\aa$.  Since $\aa'$ contains a crossing tile
of~$E$ other than~$\boxplus_2$ (by construction and transversality
of~$E$), we conclude that $\aa$ does, as well.

Assume that some box of~$A$ lies on~$\boxplus_2$. For notation, let
$\square_P$ be the box containing the only elbow tile of~$P$ in the
same row as~$\boxplus_2$, and $\square_Q$ the box containing the only
elbow tile of~$Q$ in the same column as~$\boxplus_2$. Construct $A'$
from~$A$ using one of the following rules, depending on how $A$ is
situated with respect to $P$ and~$Q$. (Some cases are covered more
than once; for example, if the next box of~$A$ strictly southwest
of~$\boxplus_2$ lies between $P$ and~$Q$ but south of the row
containing~$\square_Q$.)
\begin{itemize}
\item If the southwest box in~$A$ is on~$\boxplus_2$, or if $A$
continues southwest with its next box in a column strictly west
of~$\square_P$, then move $A$'s box on~$\boxplus_2$ west
to~$\square_P$. \item If $A$ continues southwest of~$\boxplus_2$ with
its next box in a row strictly south of~$\square_Q$, then move $A$'s
box on~$\boxplus_2$ south to lie on~$\square_Q$.
\end{itemize}
For the remaining cases, we can assume that $A$ has a box strictly
southwest of~$\boxplus_2$ but between $P$ and~$Q$ (lying on one of~
$P$ or~$Q$ is allowed).  Let $b$ be the southwest-most such box
of~$A$, and let $\bar A$ consist of the boxes of $A$ between
$\boxplus_2$ and~$b$.

\begin{itemize}
\item
Assume that $A$ continues to the west of $P$ southwest of~$b$.  Let
$G$ be a northeast grid path passing through all the boxes in $\bar A$
as in Example~\ref{ex:adiagNE}, starting with the bottom edge of the
box on~$P$ that is in the same row as~$b$, and ending with the east
edge of~$\boxplus_2$.  Applying the flipped version of wave
propagation to the region enclosed by~$G$ and~$\dn(P)$, we conclude
that we can define $A'$ by replacing ~$\bar A \cup \{\boxplus_2\}$
with an equinumerous set of elbow tiles on~$P$.

\item
If $A$ continues to the south of $Q$ after~$b$, let $G$ be a northeast
grid path passing through all the boxes in~$\bar A$ as in
Example~\ref{ex:adiagNE}, starting with the west edge of the box
on~$Q$ in the same column as~$b$, and ending with the east edge
of~$\boxplus_2$.  Applying wave propagation to the region enclosed by
$G$ and~$\up(Q)$, we conclude that we can define $A'$ by replacing
$\bar A \cup \{\boxplus_2\}$ with an equinumerous set of elbow tiles
on~$Q$.\qedhere
\end{itemize}
\end{proof}

\begin{cor}
Claim~2 holds: $E \in \AA^\vee_w \implies E \in \rp_v$ for some $v
\geq w$ \mbox{in Bruhat order}.
\end{cor}
\begin{proof}
Bruhat order is characterized by $v \geq w \iff r_{pq}(v) \leq
r_{pq}(w)$ for all $p,q$.  As \mbox{$E \in \AA^\vee_w \implies E \in
\rp_v$} for some~$v$ by Proposition~\ref{reduced}, we get $v \geq w$
by Lemma~\ref{waterfront}.
\end{proof}

%%%%%%%%%%%%%%%%%%%%%%%%%%%%%%%%%%%%%%%%%%%%%%%%%%%%%%%%%%%%%%%%%%%%%%%%
%%%%%%%%%%%%%%%%%%%%%%%%%%%%%%%%%%%%%%%%%%%%%%%%%%%%%%%%%%%%%%%%%%%%%%%%

%%%%%%%%%%%%%%%%%%%%%%%%%%%%%%%%%%%%%%%%%%%%%%%%%%%%%%%%%%%%%%%%%%%%%%%%
\end{document}